\documentclass[11pt]{article}
\usepackage{amssymb}
\usepackage{amsmath}
\usepackage[latin1]{inputenc}
\usepackage{graphicx,color}
\usepackage{tikz}
\usepackage{color}
\newtheorem{theorem}{Theorem}[section]

\newtheorem{lemma}{Lemma}[section]

\newtheorem{example}{Example}[section]

\newcommand{\fim}{\hfill\rule{2mm}{2mm}}

\begin{document}
	\title{
		\vspace{0.5in}  {\bf Existence of positive solution for a singular elliptic problem with an asymptotically linear nonlinearity }}
	
	\author{ {\bf\large Ricardo Lima Alves}\footnote{The author was supported by CNPq/Brazil Proc. $N^o$ $141110/2017-1$. }
		\hspace{2mm}\\
		{\it\small  Universidade de Bras\'ilia, Departamento de Matem\'atica}\\
		{\it\small   70910-900, Bras\'ilia - DF - Brazil}\\
		{\it\small e-mails: ricardoalveslima8@gmail.com }\vspace{1mm}\\
	}
	
	\date{}
	\maketitle \vspace{-0.2cm}

	\begin{abstract}
		In this paper we consider
		the existence of positive solutions for a singular elliptic problem involving an asymtotically linear nonlinearity and depending on one positive parameter. Using variational methods, together with comparison techniques, we show the existence, uniqueness, non-existence and regularity of the solutions. We also obtain a bifurcation-type result.
	\end{abstract}

	
	%
	%
	%
	\section{\bf Introduction}
	\def\theequation{1.\arabic{equation}}\makeatother
	\setcounter{equation}{0}

In this paper we deal  with the following semilinear elliptic problem involving a singular term: 
\begin{equation}\label{pq}\tag{$P_{\lambda}$}
	\left\{
	\begin{aligned}
		-\Delta u =a(x) u^{-\gamma} + \lambda f(u)~in ~ \Omega,\\
		u> 0~in~ \Omega,~~  u(x)=0~~on~~\partial \Omega,
	\end{aligned}
	\right.
\end{equation}
where $0<\gamma, \Omega \subset \mathbb{R}^{N}~ (N\geq 3)$ is a bounded domain with smooth boundary $\partial \Omega$ and $a(x)$ is a positive function that belongs to $L^{1}(\Omega)$. The continuous function $f:\mathbb{R}^{+} \to \mathbb{R}^{+} (\mathbb{R}^{+}= [0,\infty[) $ satisfies
\begin{itemize}
	\item[$(f)_{1}$] $\displaystyle \lim_{s\to \infty}\frac{f(s)}{s}=\theta $ for some $\theta\in (0,\infty)$,
	\item[$(f)_{2}$] $s\to \frac{f(s)}{s}$ is non-increasing in $(0,\infty)$.
\end{itemize}

We say that $u\in H_{0}^{1}(\Omega)$ is a solution of $(P)_{\lambda}$ if $u>0$ almost everywhere (a.e.) in $\Omega$, and, for every $\phi \in H_{0}^{1}(\Omega)$,
$$a(x)u^{-\gamma}\phi \in L^{1}(\Omega)$$
and 
$$\displaystyle \int_{\Omega} \nabla u\nabla \phi=\displaystyle \int_{\Omega} a(x) u^{-\gamma} \phi+\lambda \int_{\Omega}  f(u) \phi. $$

The study of singular elliptic problems started with
the pioneering work of Fulks-Maybee \cite{FM} and received a considerable attention after the paper of Crandall-Rabinowitz-Tartar \cite{CRT} (see e.g. \cite{RM, CD,H,HSS,LM,PRR,PW,CAS, SL1,SWL} and the references therein). Note that due to the presence of the singular term some difficulties appear to solve the problem $(P)_{\lambda}$. For example, problem $(P)_{\lambda}$ does not have a variational structure to apply classical results of critical point theory which are useful in the study of nonlinear boundary value problems (see e.g. \cite{ABC, BN, B}).

The problem $(P)_{\lambda}$ was studied by Anello-Faraci \cite{AF} when $a(x)\equiv 1$ and $0<\gamma<1$. By combining truncation techniques with variational methods, together with comparison techniques they proved the existence, non-existence and uniqueness of solution to $(P)_{\lambda}$ and obtained a bifurcation-type result. In this work, we complete the study done by \cite{AF} considering the case $\gamma\geq 1$ and proving new results even when $0<\gamma<1$. We would like to point out that the approach used in \cite{AF}  can not be applied when $\gamma\geq 1$.

Here we intend to use variational methods as well, but in a different way from previous works. Indeed, we give a direct method to obtain solutions of $(P)_{\lambda}$. In our approach we do not use truncation as in previous works, see for example \cite{AF, PRR, PW}. We do not invoke sets constraint to use the variational principle of Ekeland as in Sun \cite{Y} and Sun-Zhang \cite{Y1}. In fact, the technique used in \cite{Y,Y1} is more efficient when the nonlinearity $f(t)$ is homogeneous and sublinear, such as $f(t)=t^{p},~0<p<1$ and  it may not be applied to a more general nonlinearity such as the nonlinearities considered in our work.

Our study was motivated by the papers \cite{AF,Y,Y1}. Before stating our main results we will give some definitions and a general summary of our approach to prove the existence of solutions of the problem $ (P)_{\lambda} $.
First, we can associate to the problem $(P)_{\lambda}$ the following energy functional
$$I_{\lambda}(u)=\frac{1}{2}\Vert u\Vert^{2}-\lambda \displaystyle \int_{\Omega}F(u)-\displaystyle \int_{\Omega}a(x) G(\vert u\vert),$$
for every $u\in D$ (for the definition of $F$ and $G$ see section $2$), where 
$$D=\left\{u\in H^{1}_{0}(\Omega): \displaystyle \int_{\Omega}a(x)G(\vert u\vert)\in \mathbb{R} \right\}$$
is the effective domain of $I_{\lambda}$. It should be noted that for $\gamma\geq 1$, $D$ is not closed as usual (certainly not weakly closed). Indeed, if $u\in D$ then $n^{-1}u\in D$ for all $n\in \mathbb{N}$ and $n^{-1}u\to 0$ in $H_{0}^{1}(\Omega)$. Since $0\notin D$ we have that the set $D$ is not closed.

Notice that when $a(x)\in L^{\infty}(\Omega)$ and $0<\gamma<1$ we have that $D= H^{1}_{0}(\Omega)$. On the other hand, when $\gamma\geq 1$ and $a(x)\in L^{1}(\Omega)$ may occur that $D=\emptyset $. In fact, assume that $a(x)>c$ a.e. in $\Omega$ for some constant $c>0$ and $\gamma \geqslant 3$. Then from Theorem $2$ of \cite{Y1} it follows that
$$\infty=\displaystyle c\int_{\Omega}\vert u \vert^{1-\gamma}\leqslant \int_{\Omega}a(x)\vert u \vert^{1-\gamma}=(1-\gamma) \displaystyle \int_{\Omega}a(x)G(\vert u\vert)$$
for every $u\in H^{1}_{0}(\Omega)$, and  therefore we have $D=\emptyset $. Thus, to study problem $(P)_{\lambda}$ in the case of strong singularity $\gamma\geq 1$ we must assume that $D\neq\emptyset$, that is, there exists a function $u_{0}\in H_{0}^{1}(\Omega)$ such that $a(x)G(\vert u_{0}\vert)\in L^{1}(\Omega)$. 

The main difficulty to prove existence of solutions comes from case $\gamma\geq 1$ and the fact that the nonlinearity $f(t)$ is not homogeneous. In fact, in this case the effective domain $D$ is not closed and the functional $I_{\lambda}$ is not continuous. Thus the arguments used in the papers mentioned above can not be applied. Our aim is to prove that even with these difficulties that there is a global minimum $u_{\lambda}>0$ a.e. in $\Omega$ of $I_{\lambda}$. After this, using the strong singularity we have that $u_{\lambda}+\epsilon\phi \in D$ for every $\epsilon>0$, where $0\leq \phi \in H^{1}_{0}(\Omega)$ and therefore $I_{\lambda}(u_{\lambda})\leq I_{\lambda}(u_{\lambda}+\epsilon\phi)$ holds, and this inequality will help us to prove that $u_{\lambda}$ is solution of $(P)_{\lambda}$ to $\lambda$ suitable (see Theorem \ref{T1}).




We denote by $\phi_{1}$ the $L^{\infty}(\Omega)$-normalized (that is, $\Vert \phi_{1}\Vert_{\infty}=1$) positive eigenfunction for the smallest eigenvalue $\delta_{1}>0$ of  $\left(-\Delta, H_{0}^{1}(\Omega)\right)$, and set
\begin{equation}\label{3}
	\lambda_{\ast}=\frac{\delta_{1}}{\theta}.
\end{equation}

Our main results are as follows.
\begin{theorem}\label{T1}
	Assume that $ \gamma \geq 1$ and $D\neq \emptyset$. Let $\lambda_{\ast}$ be as in \eqref{3} and $a(x)\in L^{\infty}(\Omega)$ if $\gamma=1$. If $(f)_{1}-(f)_{2}$ hold, then for each $\lambda \in [0,\lambda_{\ast})$ there exists a unique solution $u_{\lambda}$ of $(P)_{\lambda}$. Moreover, $u_{\lambda}$ is a global minimum of $I_{\lambda}$, that is
	\begin{equation}\label{129}
		I_{\lambda}(u_{\lambda})=\displaystyle \inf_{u\in D}I_{\lambda}(u)
	\end{equation}
	and the following properties are satisfied:
	\begin{itemize}
		
		\item[a)] $\Vert u_{\lambda} \Vert \to \infty$ as $\lambda \to \lambda_{\ast}$,
		\item[b)] the function $(0,\lambda_{\ast})\ni \lambda \to I_{\lambda}(u_{\lambda})$  belongs to $C^{1}((0,\lambda_{\ast}),\mathbb{R})$ and 
		$$
			\frac{dI_{\lambda}(u_{\lambda})}{d\lambda}=-\displaystyle \int_{\Omega}F(u_{\lambda}),
		$$
		\item[c)] if there exist constants $C>0$ and $\eta\in (0,1)$ such that $a(x)\leq Cd^{\gamma-\eta}(x,\partial \Omega)$ a.e. in $\Omega$, then  $u_{\lambda}$  belongs to $C^{1,\alpha}(\overline{\Omega})$ for some $\alpha\in (0,1)$. The function $d(x,\partial \Omega)$ denotes the distance from a point $x\in \overline{\Omega}$ to the boundary $\partial \Omega$, where $\overline{\Omega}=\Omega \cup \partial \Omega$ is the closure of $\Omega \subset \mathbb{R}^{N}$,
		\item[d)]  assume that $f(s)=s$ for all $s\in \mathbb{R}^{+}$. If $\gamma>1$ and $\phi_{1}\in D$, then $I_{\lambda}(u_{\lambda})\to 0$ as $\lambda \to \lambda_{\ast}$.
	\end{itemize}
	
\end{theorem}

\begin{theorem}\label{T2}
	Assume that $0<\gamma<1$ and $a(x)\in L^{\infty}(\Omega)$. Let $\lambda_{\ast}$ be as in \eqref{3}. Then, for each $\lambda \in [0,\lambda_{\ast})$ there exists a unique  solution $u_{\lambda}$ of $(P)_{\lambda}$. Moreover, $u_{\lambda}$ is a global minimum of $I_{\lambda}$, that is
	\begin{equation}\label{1}
		I_{\lambda}(u_{\lambda})=\displaystyle \inf_{u\in D}I_{\lambda}(u)
	\end{equation}
	and the following properties are satisfied:
	\begin{itemize}
		\item[a)] $I_{\lambda}(u_{\lambda})\to -\infty$ as $\lambda \to \lambda_{\ast}$,
		\item[b)] $\Vert u_{\lambda} \Vert \to \infty$ as $\lambda \to \lambda_{\ast}$,
		\item[c)]if $f(t)t$ is increasing, then the function $(0,\lambda_{\ast})\ni \lambda \to \Vert u_{\lambda}\Vert$ is increasing,
		\item[d)] the function $(0,\lambda_{\ast})\ni \lambda \to I_{\lambda}(u_{\lambda})$ belongs to $C^{1}((0,\lambda_{\ast}),\mathbb{R})$ and 
		$$
			\frac{dI_{\lambda}(u_{\lambda})}{d\lambda}=-\displaystyle \int_{\Omega}F(u_{\lambda}),
		$$
		\item[e)] the function $u_{\lambda}$  belongs to $C^{1,\alpha}(\overline{\Omega})$ for some $\alpha\in (0,1)$.
	\end{itemize}
	
\end{theorem} 

Some examples of functions satisfying the assumptions $D\neq \emptyset$ and $(f)_{1}-(f)_{2}$ are as follows.
\begin{example}\begin{itemize} \item[$(a)_{1}$] The function $d(x):=d(x,\partial \Omega)$ denotes the distance from a point $x\in \overline{\Omega}$ to the boundary $\partial \Omega$, where $\overline{\Omega}=\Omega \cup \partial \Omega$ is the closure of $\Omega \subset \mathbb{R}^{N}$. Note that the strong maximum principle and boundary point principles of V\'azquez \cite{V} guarantee $\phi_{1}>0$ in $\Omega$ and $\frac{\partial \phi_{1}}{\partial \nu}
		<0$ on $\partial \Omega$, respectively. Hence, since $\phi_{1}\in C^{1}(\overline{\Omega})$ there are constants $c$ and $C$, $0<c<C$, such that $cd(x)\leq \phi_{1}(x)\leq Cd(x)$ for all $x\in \Omega$. Let $a(x):=d^{\eta}(x)$ and $\gamma>1$, where $1+\eta-\gamma>0$. Thus, there is a constant $c>0$ such that
		$$a(x)\vert \phi_{1}(x)\vert^{1-\gamma}	\leq cd^{1+\eta-\gamma}(x),$$	
		and since $\phi_{1}\in H_{0}^{1}(\Omega)$ and $cd^{1+\eta-\gamma}(x)\in L^{1}(\Omega)$, we have that $D\neq \emptyset$.	
		\item[$(a)_{2}$] Let $a(x)\equiv 1$ and $\gamma=1$. If we choose  $\eta\in (0,1)$, we obtain that $\displaystyle \lim_{t\to 0^{+}}t^{\eta}\ln t=0$. Therefore, as $\phi_{1}\in L^{\infty}(\Omega)$ there is a constant $c>0$ such that $\vert \phi_{1}^{\eta}(x)\ln \phi_{1}(x)\vert \leq c$ for all $x\in \Omega$, and from \cite{LM} we have that 
		$$\vert \ln \phi_{1}(x)\vert= \vert \phi_{1}^{\eta}(x)\ln \phi_{1}(x)\vert \phi_{1}^{-\eta}(x)\leq c \phi_{1}^{-\eta}(x) \in L^{1}(\Omega),$$
		and this implies that $D\neq \emptyset$.	
		\item[$(e)_{1}$] The function $f(s)=s,~s\geq 0$ satisfies $(f)_{1}-(f)_{2}$. In this case, we have that $\lambda_{\ast}=\lambda_{1}$.
		
		\item[$(e)_{2}$] The function $f(s)=as+s^{r}+1$, $s\geq 0$ $(r\in (0,1))$, where $a>0$ satisfies $(f)_{1}-(f)_{2}$. In this case, we have that $\lambda_{\ast}=\lambda_{1}/a$.
	\end{itemize}
\end{example}

The Theorem \ref{T1} extends the main result of Anello-Faraci \cite{AF} (see Theorem 2.2 in  \cite{AF}) in the sense that we consider strong singularity $\gamma\geq 1$. As far as we know, the properties $a)-d)$ of Theorem \ref{T1} and $a)-e)$ of Theorem \ref{T2} are new for singular nonlinearity. We prove in particular that $\lambda_{\ast}$ is a "bifurcation point from infinity" of $(P)_{\lambda}$. See the graphs below.
\begin{figure}[h]
	\centering
	\begin{tikzpicture}[scale=.65]
	\draw[thick, ->] (-1, 1) -- (6, 1);
	\draw[thick, ->] (0, 0) -- (0, 7);
	\draw[thick] (0, 2) .. controls (0, 2) and (4, 4) ..(3.85,7);
	\draw (0,7) node[left]{$\mathbb{R}$};
	\draw (6,-0.2) node[above]{$\mathbb{R}$};
	\draw (2,-0.2) node[above]{$\lambda$};
	\draw (0, 1) node[below left]{$0$};
	\draw[thick, dotted] (2, 1) -- (2,3.5);
	\draw (4,-0.2) node[above]{$\lambda_{\ast}$};
	\draw (1.5,3.4) node[above]{$\Vert u_{\lambda}\Vert $};
	\draw[thick, dotted] (4, 1) -- (4,7);
	\draw (3,0) node[below]{\textrm{Fig.1. Theorems \ref{T1} and \ref{T2} }};
	\end{tikzpicture}
	\hspace{0.8cm}
	\begin{tikzpicture}[scale=.65]
	\draw[thick, ->] (-1, 1) -- (6, 1);
	\draw[thick, ->] (0, 0) -- (0, 7);
	\draw[thick] (0, 5) .. controls (0, 5) and (3, 3) ..(4,1);
	\draw (0,7) node[left]{$\mathbb{R}$};
	\draw (6,-0.2) node[above]{$\mathbb{R}$};
	\draw (2,-0.2) node[above]{$\lambda$};
	\draw (0, 1) node[below left]{$0$};
	\draw[thick, dotted] (2, 1) -- (2,3.4);
	\draw (4,-0.2) node[above]{$\lambda_{\ast}$};
	\draw (2.55,3.4) node[above]{$I_{\lambda}(u_{\lambda}) $};
	\draw[thick, dotted] (4, 1) -- (4,7);
	\draw (3,0) node[below]{\textrm{Fig.2. Theorem \ref{T1} }};
	\end{tikzpicture}
	\hspace{0.8cm}
	\begin{tikzpicture}[scale=.65]
	\draw[thick, ->] (-1, 3) -- (6, 3);
	\draw[thick, ->] (0, 0) -- (0, 7);
	\draw[thick] (0, 5) .. controls (0, 5) and (3, 3) ..(3.85,0);
	\draw (0,7) node[left]{$\mathbb{R}$};
	\draw (6,2.2) node[above]{$\mathbb{R}$};
	\draw (1,3) node[below]{$\lambda$};
	\draw (0, 3) node[below left]{$0$};
	\draw[thick, dotted] (1, 3) -- (1,4.24);
	\draw (4.5,3) node[below]{$\lambda_{\ast}$};
	\draw (2,4) node[above]{$I_{\lambda}(u_{\lambda}) $};
	\draw[thick, dotted] (4, 0) -- (4,7);
	\draw (3,0) node[below]{\textrm{Fig.3. Theorem \ref{T2} }};
	\end{tikzpicture}
\end{figure}



The paper is organized as follows. Section 1 is devoted to some preliminaries and the existence of global minimum of $I_{\lambda}$ for $\lambda\in [0,\lambda_{\ast})$. In Section 2 we prove the Theorem \ref{T1}. In Section 3 we prove the Theorem \ref{T2}.

\textbf{Notation} Throughout this paper, we make use of the following notation:\\
\begin{itemize}
	\item $L^{p}(\Omega)$, for $1\leq p \leq \infty$, denotes the Lebesgue space with usual norm denoted by $\vert u\vert_{p}$.
	\item $H_{0}^{1}(\Omega)$ denotes the Sobolev space endowed with inner product
	$$\left(u,v\right)_{H}=\displaystyle \int_{\Omega} \nabla u \nabla v,~\forall u,v \in H_{0}^{1}(\Omega). $$
	The norm associated with this inner product will be denoted by $\Vert~~ \Vert$. 
	\item If $u$ is a measurable function, we denote by $u^{+}$ the positive part of $u$, which is given by $u^{+}=\left\{u,0\right\}$. 
	\item If $A$ is a measurable set in $\mathbb{R}^{N}$, we denote by $\vert A \vert$ the Lebesgue measure of $A$. 
	\item We denote by $\phi_{1}$ the $L^{\infty}(\Omega)$-normalized positive eigenfunction for the smallest eigenvalue $\delta_{1}>0$ of  $\left(-\Delta, H_{0}^{1}(\Omega)\right)$. 
	\item The function $d(x):=d(x,\partial \Omega)$ denotes the distance from a point $x\in \overline{\Omega}$ to the boundary $\partial \Omega$, where $\overline{\Omega}=\Omega \cup \partial \Omega$ is the closure of $\Omega \subset \mathbb{R}^{N}$.
	\item $c$ and $C$ denote (possibly different from line to line) positive constants.
\end{itemize}

\section{Existence of global minimum and Preliminaries}
This section deals with the existence of global minimum of $I_{\lambda}$ over the effective domain $D$, for each $\lambda \in [0,\lambda_{\ast})$. We also give some preliminary results. First, let us introduce the energy functional associated to the problem $(P)_{\lambda}$ and some of its properties. Define the function $G$ as it follows:\\
if $0<\gamma<1$, $G(t)= \frac{\vert t\vert^{1-\gamma}}{1-\gamma}$ and $t\in \mathbb{R}$,\\
if $\gamma=1$, \begin{equation*}
	G(t)=\left\{
	\begin{array}{l}
		\ln t,~if ~ t> 0\\
		+\infty,~if~ t=0
	\end{array}
	\right.
\end{equation*}
if $\gamma>1$,
\begin{equation*}
	G(t)=\left\{
	\begin{array}{l}
		\frac{t^{1-\gamma}}{1-\gamma},~if ~ t> 0\\
		+\infty,~if~ t=0.
	\end{array}
	\right.
\end{equation*}

Let 
\begin{equation*}
	f_{0}(t)=\left\{
	\begin{array}{l}
		f(t),~if ~ t\geq 0\\
		f(0),~if~ t<0.
	\end{array}
	\right.
\end{equation*}

From now on we will assume that:
\begin{itemize}
	\item  $0<a(x)\in L^{\infty}(\Omega)$, if $0<\gamma \leq 1$,
	\item  $0<a(x)\in L^{1}(\Omega)$, if $\gamma >1$.
\end{itemize}

So, we can associate to the problem $(P)_{\lambda}$ the following energy functional 
$$I_{\lambda}(u)=\frac{1}{2}\Vert u\Vert^{2}-\lambda \displaystyle \int_{\Omega}F(u)-\displaystyle \int_{\Omega}a(x)G(\vert u \vert )$$
for $u\in D$, where 
$$D=\left\{u\in H^{1}_{0}(\Omega): \displaystyle \int_{\Omega}a(x)G(\vert u \vert)\in \mathbb{R} \right\}$$ 
is the effective domain of $I_{\lambda}$ and 
$$F(t)=\displaystyle \int_{0}^{t}f_{0}(s)ds.$$

Now, by $(f)_{1}-(f)_{2}$ for every $\epsilon>0$ there exist constants $c=c(\epsilon)$ and $C=C(\epsilon)$  such that
\begin{equation}\label{15}
	\theta s \leq f(s)\leq \left(\theta +\epsilon\right)s+c,~\forall s\geq 0,
\end{equation}

\begin{equation}\label{8}
	\theta \frac{s^{2}}{2}\leq	F(s)\leq \frac{\left(\theta +\epsilon\right)}{2}s^{2}+Cs,~\forall s\geq 0.
\end{equation}

Assuming that $a(x)\in L^{\infty}(\Omega)$ and $0<\gamma<1$, we have that $D=H^{1}_{0}(\Omega)$ and the functional $I_{\lambda}$ is continuous, but it is not G\^ateaux differentiable in $D=H_{0}^{1}(\Omega)$. On the other hand, when $1\leq\gamma$ we have $D\subsetneqq H^{1}_{0}(\Omega)$. Since we are interested in positive solutions, we introduce the set 
$$D^{+} = \left\{u\in D: u\geq 0~\mbox{a.e. in}~\Omega\right\}.$$
and let us show that 
$$\displaystyle \inf_{v\in D}I_{\lambda}(v)=\displaystyle \inf_{v\in D^{+}}I_{\lambda}(v).$$

For each $u\in D$, we have
\begin{align*}
	I_{\lambda}(u)&=\frac{1}{2}\Vert \vert u\vert \Vert^{2}-\lambda \displaystyle \int_{u<0}F(u)-\lambda \displaystyle \int_{u\geq 0}F(u)-\displaystyle \int_{\Omega}a(x)G (\vert u\vert)\\
	&=\frac{1}{2}\Vert \vert u\vert \Vert^{2}-\lambda \displaystyle \int_{u<0}f(0)u-\lambda \displaystyle \int_{u\geq 0}F(u)-\displaystyle \int_{\Omega}a(x)G (\vert u\vert)\\
	&\geq \frac{1}{2}\Vert \vert u\vert \Vert^{2}-\lambda \displaystyle \int_{\Omega}F(\vert u\vert)-\displaystyle \int_{\Omega}a(x)G (\vert u\vert)= I_{\lambda}(\vert u\vert ),
\end{align*}
and this implies that $u\in D$ satisfies $I_{\lambda}(u)=\displaystyle \inf_{v\in D}I_{\lambda}(v)$ if and only if $\vert u\vert \in D^{+}$ satisfies $I_{\lambda}(\vert u \vert)=\displaystyle \inf_{v\in D^{+}}I_{\lambda}(v)$. Thus,
$$I_{\lambda}(u)=\displaystyle \inf_{v\in D}I_{\lambda}(v)=\displaystyle \inf_{v\in D^{+}}I_{\lambda}(v)= I_{\lambda}(\vert u \vert)$$
and to show \eqref{129} and \eqref{1} it is sufficient show that there exists $u\in D^{+}$ such that $I_{\lambda}( u)=\displaystyle \inf_{v\in D^{+}}I_{\lambda}(v)$.

After the considerations and definitions above we have the following lemma. It provides the existence of a global minimum of $I_{\lambda}$ for every $\lambda \in (0,\lambda_{\ast})$.

\begin{lemma}\label{l1}
	Assume that $0\leq \lambda<\lambda_{\ast}$. Then there exists $u_{\lambda}\in D^{+}$ such that 
	$$I_{\lambda}(u_{\lambda})=\displaystyle \inf_{u\in D^{+}}I_{\lambda}(u).$$
\end{lemma}
\begin{proof}
	Let $\epsilon>0$ such that
	$$\frac{\lambda\left(\theta+\epsilon\right)}{\delta_{1}}<1.$$
	
	Using \eqref{8}, Sobolev embedding  and Poincaré inequality we have
	\begin{equation}\label{2}
		I_{\lambda}(u)\geq \frac{1}{2}\left(1-\frac{\lambda\left(\theta+\epsilon\right)}{\delta_{1}}\right)\Vert u\Vert^{2}-c\Vert u\Vert-\displaystyle \int_{\Omega}a(x)G(\vert u \vert),
	\end{equation} 
	for every $u\in D^{+}$. Now, we have three cases to consider.\\
	\textbf{Case 1.} $0<\gamma <1$. In this case, by \eqref{2} the functional $I_{\lambda}$  is coercive and, since that $I_{\lambda}$ is sequentially weakly lower semicontinuous in $D=H^{1}_{0}(\Omega)$ a usual argument proves the statement of the lemma.\\ 
	\textbf{Case 2.}  $\gamma=1$. Since $a(x)\in L^{\infty}(\Omega)$ and $\ln \vert u\vert \leq \vert u\vert$, from \eqref{2} and Sobolev embedding we obtain
	$$
		I_{\lambda}(u)\geq \frac{1}{2}\left(1-\frac{\lambda\left(\theta+\epsilon\right)}{\delta_{1}}\right)\Vert u\Vert^{2}-C\Vert u\Vert
	$$
	for every $u\in D^{+}$, which implies that $I_{\lambda}$ is coercive in $ D^{+}$. Let $\left\{  u_{n} \right\}\subset D^{+}$ be a sequence such that $I_{\lambda}(u_{n})\to \displaystyle \inf_{u\in D^{+}}I_{\lambda}(u)\in \mathbb{R}$ as $n\to \infty$. So, we may assume that there is $0\leq u_{\lambda}\in H^{1}_{0}(\Omega)$ such that $u_{n}\rightharpoonup u_{\lambda}$ in $H^{1}_{0}(\Omega)$, $u_{n}\to u_{\lambda}$ in $L^{p}(\Omega),~p\in (0,2^{\ast})$ and $u_{n}\to u_{\lambda}$ a.e. in $\Omega$. Using the fact that $\displaystyle \inf_{u\in D^{+}}I_{\lambda}(u)\in \mathbb{R}$  we obtain that the sequence 
	$\left\{ \int_{\Omega}a(x)G(\vert u_{n} \vert)\right\}=\left\{ \int_{\Omega}a(x)\ln \vert u_{n} \vert\right\}$ is bounded, which implies that $\displaystyle \limsup_{n\to \infty}\int_{\Omega}a(x)\ln \vert u_{n} \vert \in \mathbb{R}$. Since  $\ln \vert u_{n}\vert \leq \vert u_{n}\vert$ for all $n$, Fatou's lemma yields 
	$$-\infty< \displaystyle \limsup_{n\to \infty}\int_{\Omega}a(x)\ln \vert u_{n} \vert\leq \displaystyle\int_{\Omega} \limsup_{n\to \infty}a(x)\ln \vert u_{n} \vert=\displaystyle \int_{\Omega}a(x)\ln \vert u_{\lambda}\vert$$
	and this allows us to conclude that $u_{\lambda}\in D^{+}$ and 
	\begin{align*}
		\inf_{u\in D^{+}}I_{\lambda}(u)=&\displaystyle \liminf_{n\to \infty}I_{\lambda}(u_{n})\geq \frac{1}{2}\Vert u_{\lambda}\Vert^{2}-\lambda \displaystyle \int_{\Omega}F(u_{\lambda})-\displaystyle \limsup_{n\to \infty}  \int_{\Omega}a(x)\ln \vert u_{n} \vert\\
		\geq &  \frac{1}{2}\Vert u_{\lambda}\Vert^{2}-\lambda \displaystyle \int_{\Omega}F(u_{\lambda})-\displaystyle \int_{\Omega}a(x)\ln \vert u_{\lambda} \vert=I_{\lambda}(u_{\lambda})\geq \inf_{u\in D^{+}}I_{\lambda}(u_{\lambda}),
	\end{align*}
	that is $I_{\lambda}(u_{\lambda})= \inf_{u\in D^{+}}I_{\lambda}(u)$.\\
	\textbf{Case 3.} $\gamma>1$. From \eqref{2} we have
	\begin{equation}\label{20}
		I_{\lambda}(u)\geq \frac{1}{2}\left(1-\frac{\lambda\left(\theta+\epsilon\right)}{\delta_{1}}\right)\Vert u\Vert^{2}-c\Vert u \Vert -\displaystyle \frac{1}{1-\gamma}\int_{\Omega}a(x)\vert u \vert^{1-\gamma}\geqslant \frac{1}{2}\left(1-\frac{\lambda\left(\theta+\epsilon\right)}{\delta_{1}}\right)\Vert u\Vert^{2}-c\Vert u \Vert ,
	\end{equation}
	for all $u\in D^{+}$, which implies that $I_{\lambda}$ is coercive in $ D^{+}$. Let $\left\{  u_{n} \right\}\subset D^{+}$ be a sequence such that $I_{\lambda}(u_{n})\to \displaystyle \inf_{u\in D^{+}}I_{\lambda}(u)\in \mathbb{R}$ as $n\to \infty$. So, we may assume that there is $0\leq u_{\lambda}\in H^{1}_{0}(\Omega)$ such that $u_{n}\to u_{\lambda}$ in $H^{1}_{0}(\Omega)$, $u_{n}\rightharpoonup u_{\lambda}$ in $L^{p}(\Omega),~p\in (0,2^{\ast})$ and $u_{n}\to u_{\lambda}$ a.e. in $\Omega$. From \eqref{20} and Fatou's lemma we obtain
	$$\infty > \displaystyle \liminf I_{\lambda}(u_{n})\geqslant \frac{1}{2}\Vert u_{\lambda}\Vert^{2}-\lambda \displaystyle \int_{\Omega}F(u_{\lambda})-\dfrac{1}{1-\gamma}\displaystyle \int_{\Omega}a\vert u_{\lambda} \vert^{1-\gamma},$$
	and as a consequence of this we have that $u_{\lambda}\in D^{+}$ and $I_{\lambda}(u_{\lambda})=\inf_{u\in D^{+}}I_{\lambda}(u)$. The proof is complete.
	
	\fim
\end{proof}

To study the non-existence, uniqueness and asymptotic behavior of our main results we will need some auxiliary results.

Our first auxiliary result is a non-existence lemma. It proves that problem $(P)_{\lambda}$ has no solution for $\lambda \geq \lambda_{\ast}$.
\begin{lemma}\label{l0}
	Assume that problem $(P)_{\lambda}$ has solution. Then $0\leq \lambda<\lambda_{\ast}$ .
\end{lemma}
\begin{proof}
	First we know by \eqref{15} that $f(t)\geq \theta t$ for all $t>0$. Assume that $u$ is a solution of $(P)_{\lambda}$. Thus, we have
	$$\delta_{1}\displaystyle \int_{\Omega}  u \phi_{1}=\displaystyle \int_{\Omega} \nabla u\nabla \phi_{1}=\displaystyle \int_{\Omega} a(x) u^{-\gamma} \phi_{1}+\lambda \int_{\Omega}  f(u) \phi_{1} >\theta \lambda \int_{\Omega}  u \phi_{1} ,
	$$
	which implies that $\lambda_{\ast}=\frac{\delta_{1}}{\theta}>\lambda$. The proof is complete.

\fim	
\end{proof}

Related the uniqueness of the solutions we will need the following comparison lemma.

\begin{lemma}\label{l3}
	Let $\underline{u},\overline{u}\in H_{0}^{1}(\Omega), \underline{u},\overline{u}>0$ in $\Omega$, such that $a(x)\underline{u}^{-\gamma}\phi,a(x)\overline{u}^{-\gamma}\phi\in L^{1}(\Omega)$ for every $\phi \in H_{0}^{1}(\Omega)$ and such that:
	\begin{equation}\label{c1}
		\displaystyle \int_{\Omega} \nabla \underline{u} \nabla \phi \leq  \displaystyle \int_{\Omega}a(x) \underline{u}^{-\gamma}\phi+\lambda \displaystyle \int_{\Omega} f(\underline{u})\phi ~\mbox{for every}~\phi \in H_{0}^{1}(\Omega),\phi \geq 0,
	\end{equation}
	\begin{equation}\label{c2}
		\displaystyle \int_{\Omega} \nabla \overline{u} \nabla \phi \geq \displaystyle \int_{\Omega} a(x)\overline{u}^{-\gamma}\phi+\lambda \displaystyle \int_{\Omega} f(\overline{u})\phi ~\mbox{for every}~\phi \in H_{0}^{1}(\Omega),\phi \geq 0.
	\end{equation}
	Then, $\underline{u}\leq \overline{u}$ a.e. in $\Omega$. 
\end{lemma}
\begin{proof}
	We follow the arguments of \cite{AF}. Choose a funtion $\sigma:\mathbb{R}\to [0,\infty)$, non-decreasing, of class $C^{1}(\mathbb{R})$ such that $\sigma(t)=0$ for $t\leq 0$, $\sigma(t)=1$ for $t\geq 1$. For any $\varepsilon>0$ put $\sigma_{\epsilon}(t)=\sigma(t/\varepsilon)$. Then $\sigma_{\epsilon}$ is a $C^{1}$ function and there exists a constant $c>0$ such that $\sigma^{\prime}_{\varepsilon}(t)\leq c/\varepsilon$ for any $t\in \mathbb{R}$. Let us denote by 
	$$h(x,t)= a(x)t^{-\gamma}+\lambda f(t)~\mbox{for any}~t\in \mathbb{R}~\mbox{and}~x\in \Omega.$$
	
	Choosing as test function in \eqref{c1} and \eqref{c2},
	$$\phi_{1,\varepsilon}(x)=\overline{u}(x)\sigma_{\varepsilon}(\underline{u}(x)-\overline{u}(x))$$
	and
	$$\phi_{2,\varepsilon}(x)=\underline{u}(x)\sigma_{\varepsilon}(\underline{u}(x)-\overline{u}(x)),$$
	respectively, and subtracting \eqref{c1} from \eqref{c2} we have
	\begin{align}\label{128}
		\nonumber &\displaystyle \int_{\Omega} \nabla \overline{u} \nabla \underline{u}\sigma_{\varepsilon}(\underline{u}-\overline{u})+\underline{u}\sigma^{\prime}_{\varepsilon}(\underline{u}-\overline{u})\nabla \overline{u} \nabla (\underline{u}-\overline{u})\\
		\nonumber &-\displaystyle \int_{\Omega} \nabla \underline{u} \nabla \overline{u}\sigma_{\varepsilon}(\underline{u}-\overline{u})+\overline{u}\sigma^{\prime}_{\varepsilon}(\underline{u}-\overline{u})\nabla \underline{u} \nabla (\underline{u}-\overline{u})\\
		&\geq \displaystyle \int_{\Omega}\left(h(x,\overline{u})\underline{u}-h(x,\underline{u})\overline{u}\right)\sigma_{\varepsilon}(\underline{u}-\overline{u}).
	\end{align}
	
	The left-hand side in \eqref{128} can be written in the following way:
	\begin{align*}
		&\displaystyle \int_{\Omega} \sigma^{\prime}_{\varepsilon}(\underline{u}-\overline{u})  \left( \nabla \overline{u} \underline{u} -\nabla \underline{u}\overline{u}\right)\nabla \left( \underline{u}-\overline{u}\right)\\
		&=  \displaystyle \int_{\Omega}\underline{u}\sigma^{\prime}_{\varepsilon}(\underline{u}-\overline{u})\nabla \left(\overline{u}-\underline{u} \right)\nabla \left(\underline{u}-\overline{u} \right)+\displaystyle \int_{\Omega}\left(\underline{u}-\overline{u} \right)\sigma^{\prime}_{\varepsilon}(\underline{u}-\overline{u})\nabla \underline{u}\nabla \left(\underline{u}-\overline{u} \right)\\
		&=-\displaystyle \int_{\Omega}\underline{u}\sigma^{\prime}_{\varepsilon}(\underline{u}-\overline{u})\vert \underline{u}-\overline{u}\vert +\displaystyle \int_{\Omega}\left(\underline{u}-\overline{u} \right)\sigma^{\prime}_{\varepsilon}(\underline{u}-\overline{u})\nabla \underline{u}\nabla \left(\underline{u}-\overline{u} \right)&\\
		&\leq \frac{c}{\varepsilon}\displaystyle \int_{\left\{0<\underline{u}-\overline{u}<\varepsilon\right\}}\left(\underline{u}-\overline{u}\right)\nabla \underline{u}\nabla \left(\underline{u}-\overline{u} \right)\\
		& \leq c \displaystyle \int_{\left\{0<\underline{u}-\overline{u}<\varepsilon\right\}}\vert \nabla \underline{u}\vert \vert \nabla \left(\underline{u}-\overline{u}\right)\vert
	\end{align*}
	and since $\vert \left\{0<\underline{u}-\overline{u}<\varepsilon\right\}\vert \to 0$ as $\varepsilon \to 0^{+}$, we obtain
	$$\displaystyle \int_{\left\{0<\underline{u}-\overline{u}<\varepsilon\right\}}\vert \nabla \underline{u}\vert \vert \nabla \left(\underline{u}-\overline{u}\right)\vert\to 0$$
	as $\varepsilon \to 0$, which implies 
	$$\displaystyle \int_{\Omega} \sigma^{\prime}_{\varepsilon}(\underline{u}-\overline{u})  \left( \nabla \overline{u} \underline{u} -\nabla \underline{u}\overline{u}\right)\nabla \left( \underline{u}-\overline{u}\right)\to 0 $$
	as $\varepsilon \to 0$. Notice that if $t>0$, then $\sigma_{\varepsilon}(t)\to 1$ when $\varepsilon \to 0$, while $\sigma_{\varepsilon}(t)=0$ if $t\leq 0$. Therefore, by using the Fatou's lemma in \eqref{128} we obtain
	\begin{align*}
		&0\geq \displaystyle\liminf_{\varepsilon\to 0} \displaystyle \int_{\Omega}\underline{u}\overline{u}\left(\frac{h(x,\overline{u})}{\overline{u}}-\frac{h(x,\underline{u})}{\underline{u}}\right)\sigma_{\varepsilon}(\underline{u}-\overline{u})\\
		&=\displaystyle\liminf_{\varepsilon\to 0}\displaystyle \int_{\left\{0<\underline{u}-\overline{u}\right\}}\underline{u}\overline{u}\left(\frac{h(x,\overline{u})}{\overline{u}}-\frac{h(x,\underline{u})}{\underline{u}}\right)\sigma_{\varepsilon}(\underline{u}-\overline{u})\\
		&\geq \displaystyle \int_{\left\{0<\underline{u}-\overline{u}\right\}}\underline{u}\overline{u}\left(\frac{h(x,\overline{u})}{\overline{u}}-\frac{h(x,\underline{u})}{\underline{u}}\right)\geq 0
	\end{align*}
	and this implies that $\vert \left\{0<\underline{u}-\overline{u}\right\}\vert=0$, that is, $\underline{u}\leq \overline{u}$ a.e. in $\Omega$. The proof is complete. 
	
	\fim
\end{proof}

Let us define
$$H_{\lambda}(u)=\Vert u\Vert^{2}-\lambda \theta\displaystyle \int_{\Omega}u^{2},$$
for every $u\in H^{1}_{0}(\Omega)$ and $\lambda>0$. On account of the Poincaré inequality, we have
$$
	H_{\lambda}(u)> \Vert u\Vert^{2}-\frac{\lambda_{1}}{\theta}\frac{\theta}{\lambda_{1}}\Vert u\Vert^{2}=0,
$$
for each $\lambda \in (0,\lambda_{\ast})$ and $u\in D^{+}$.

Thus, if $\gamma \neq 1$ from \eqref{8} we have
\begin{equation}\label{9}
	I_{\lambda}(u)<J_{\lambda}(u),
\end{equation}
for each $\lambda \in (0,\lambda_{\ast})$ and $u\in D^{+}$, where $J_{\lambda}:D\to \mathbb{R}$ is defined by
$$J_{\lambda}(u)= \frac{1}{2}H_{\lambda}(u)-\dfrac{1}{1-\gamma}\displaystyle \int_{\Omega}a(x)\vert u \vert^{1-\gamma}.$$

To study the behavior of function $(0,\lambda_{\ast})\ni \lambda \to I_{\lambda}(u_{\lambda})$ when $\lambda \to \lambda_{\ast}$, we will need the following lemma.
\begin{lemma}\label{l2}
	Assume $\gamma\neq 1$. Moreover, assume that $\lambda \in (0,\lambda_{\ast})$ and $u\in D^{+}$. Then, there exists $t_{\lambda}(u)>0$ such that 
	\begin{equation}\label{10}
		J_{\lambda}(t_{\lambda}(u)u)=\displaystyle -\left(\frac{1+\gamma}{2(1-\gamma)}\right)\left(\displaystyle \int a(x)\vert u\vert^{1-\gamma}\right)\left(\frac{\displaystyle \int a(x)\vert u\vert^{1-\gamma}}{H_{\lambda}(u)}\right)^{\frac{1-\gamma}{1+\gamma}}.
	\end{equation}
	
	In particular, if $\phi_{1} \in D^{+}$ and $\lambda \in (0,\lambda_{\ast})$, then
	\begin{equation}\label{11}
		I_{\lambda}(u_{\lambda})\leq J_{\lambda}(t_{\lambda}(\phi_{1})\phi_{1})
	\end{equation}
	holds.
\end{lemma}
\begin{proof}
	For each $\lambda \in (0,\lambda_{\ast})$ and $u\in D^{+}$ let us define the function $\psi\in C^{1}((0,\infty),\mathbb{R})$ by $\psi(t)=J_{\lambda}(tu)$. It is easy to see that there exists a $t_{\lambda}(u)$ such that $\psi(t_{\lambda}(u))=\displaystyle \inf_{t\in (0,\infty)}\psi(t)$ and
	$$t_{\lambda}(u)=\left(\frac{\displaystyle \int a(x)\vert u\vert^{1-\gamma}}{H_{\lambda}(u)}\right)^{\frac{1}{1+\gamma}}.$$
	
	Moreover, 
	$$J_{\lambda}(t_{\lambda}(u)u)=\psi(t_{\lambda}(u))=\displaystyle -\left(\frac{1+\gamma}{2(1-\gamma)}\right)\left(\displaystyle \int a(x)\vert u\vert^{1-\gamma}\right)\left(\frac{\displaystyle \int a(x)\vert u\vert^{1-\gamma}}{H_{\lambda}(u)}\right)^{\frac{1-\gamma}{1+\gamma}}$$
	holds. 
	
	Then, by Lemma \ref{l1} and \eqref{9} follows that
	$$I_{\lambda}(u_{\lambda})\leq I_{\lambda}(t_{\lambda}(\phi_{1})\phi_{1})\leq J_{\lambda}(t_{\lambda}(\phi_{1})\phi_{1}),$$
	and this show \eqref{11}. The proof is complete.
	
	\fim
	
\end{proof}
\section{Proof of Theorem \ref{T1}}

In this section we prove the Theorem \ref{T1}. Let $u_{\lambda}$ be as in the Lemma \ref{l1}. Let us prove that $u_{\lambda}$ is a solution of $(P)_{\lambda}$. To this aim, first we consider $\phi \in H^{1}_{0}(\Omega)$ such that $\phi \geq 0$ in $\Omega$ and $\epsilon>0$ and we will show that $ u_{\lambda}+\epsilon \phi\in D^{+}$.  Now, we have two cases to consider.\\
\textbf{Case 1.} $\gamma =1$ and $a\in L^{\infty}(\Omega)$. In this case, we have $G(t)=\ln t$, for all $t>0$ and 
$$-\infty<\displaystyle \int_{\Omega}a(x)\ln(u_{\lambda})\leq \displaystyle \int_{\Omega}a(x)\ln(u_{\lambda}+\epsilon \phi)\leq \displaystyle \int_{\Omega}a(x)(u_{\lambda}+\epsilon \phi)<\infty, $$
that is $ u_{\lambda}+\epsilon \phi\in D^{+}$.\\
\textbf{Case 2.} $\gamma >1$. Since $\gamma>1$ and $u_{\lambda}+\epsilon \phi\geq u_{\lambda}> 0$ we have
$$\displaystyle \int_{\Omega}a(x)\vert u_{\lambda}+\epsilon \phi\vert^{1-\gamma}\leq \displaystyle \int_{\Omega}a(x)\vert u_{\lambda}\vert^{1-\gamma}<\infty,$$
that is $ u_{\lambda}+\epsilon \phi\in D^{+}$. 

Therefore, in both cases it follows that
$$\frac{1}{2}\Vert u_{\lambda}+\epsilon \phi\Vert^{2}-\frac{1}{2}\Vert u_{\lambda}\Vert^{2}-\lambda \displaystyle \int_{\Omega}F(u_{\lambda}+\epsilon \phi)-\displaystyle F(u_{\lambda})\geq\displaystyle \int_{\Omega}a(x)G( u_{\lambda}+\epsilon \phi) -a(x)G(u_{\lambda}).$$

Thus, dividing the last inequality by $\epsilon$ and passing to the liminf as $\epsilon \to 0$, from Fatou's Lemma we have
\begin{equation}\label{12}
	\displaystyle \int_{\Omega} \nabla u_{\lambda} \nabla \phi  - \lambda\displaystyle \int_{\Omega} f(u_{\lambda})\phi \geq \displaystyle \int_{\Omega}a(x) u_{\lambda}^{-\gamma}\phi,
\end{equation} 
for every $\phi \in H^{1}_{0}(\Omega)$, with $\phi \geq 0$. 

Now, since $tu_{\lambda}\in D^{+}$ for every $t>0$, the function $\psi\in C^{1}((0,\infty),\mathbb{R})$ defined by $\psi(t)=I_{\lambda}(tu_{\lambda})$ has a global minimum at $t=1$. Therefore
\begin{equation}\label{13}
	0=\psi^{\prime}(1)=\Vert u_{\lambda}\Vert^{2}-\lambda \displaystyle \int_{\Omega}f(u_{\lambda})u_{\lambda}-\displaystyle \int_{\Omega}a(x)\vert u \vert^{1-\gamma}.
\end{equation}

Finally we use an argument inspired by Grah-Eagle \cite{GE} to prove that $u_{\lambda}$ is a solution of $(P)_{\lambda}$. Set $\Psi(x)= \left(u_{\lambda}(x)+\epsilon \phi(x)\right)^{+}$, for $\phi \in H^{1}_{0}(\Omega)$ and $\epsilon>0$. From \eqref{12} and \eqref{13} we have
\begin{align*}
	 0\leq \displaystyle \int_{\Omega} \nabla u_{\lambda} \nabla \Psi  - \lambda\displaystyle \int_{\Omega} f(u_{\lambda})\Psi - \displaystyle \int_{\Omega} a(x)u_{\lambda}^{-\gamma}\Psi &\\
	=\displaystyle \int_{\left\{u_{\lambda}+\epsilon \phi\geq 0\right\}} \nabla u_{\lambda} \nabla (u_{\lambda}+\epsilon \phi)  - \lambda f(u_{\lambda})(u_{\lambda}+\epsilon \phi) -  a(x)u_{\lambda}^{-\gamma}(u_{\lambda}+\epsilon \phi)&\\
	=\displaystyle \int_{\Omega}-\displaystyle \int_{\left\{u_{\lambda}+\epsilon \phi< 0\right\}}\nabla u_{\lambda} \nabla (u_{\lambda}+\epsilon \phi)  - \lambda f(u_{\lambda})(u_{\lambda}+\epsilon \phi) -  a(x)u_{\lambda}^{-\gamma}(u_{\lambda}+\epsilon \phi)\\
	=\psi^{\prime}(1)+\epsilon \left[\displaystyle \int_{\Omega} \nabla u_{\lambda} \nabla \phi  - \lambda f(u_{\lambda})\phi -  a(x)u_{\lambda}^{-\gamma}\phi\right]\\
	-\left[\displaystyle \int_{\left\{u_{\lambda}+\epsilon \phi< 0\right\}} \nabla u_{\lambda} \nabla (u_{\lambda}+\epsilon \phi)  - \lambda f(u_{\lambda})(u_{\lambda}+\epsilon \phi) -  a(x)u_{\lambda}^{-\gamma}(u_{\lambda}+\epsilon \phi)\right]\\
	\leq \epsilon \left[\displaystyle \int_{\Omega} \nabla u_{\lambda} \nabla \phi  - \lambda f(u_{\lambda})\phi -  a(x)u_{\lambda}^{-\gamma}\phi\right]-\epsilon \displaystyle \int_{\left\{u_{\lambda}+\epsilon \phi< 0\right\}} \nabla u_{\lambda} \nabla \phi, 
\end{align*}
and since $\vert \left\{u_{\lambda}+\epsilon \phi< 0\right\} \vert \to 0$ as $\epsilon \to 0^{+}$, dividing by $\epsilon$ and letting $\epsilon \to 0^{+}$  we obtain 
$$0\leq \displaystyle \int_{\Omega} \nabla u_{\lambda} \nabla \phi  - \lambda f(u_{\lambda})\phi - a(x) u_{\lambda}^{-\gamma}\phi,$$
for every $\phi \in H_{0}^{1}(\Omega)$. Hence, this inequality also holds equally well for $-\phi$. Thus, we have
$$ \displaystyle \int_{\Omega} \nabla u_{\lambda} \nabla \phi  - \lambda f(u_{\lambda})\phi -  au_{\lambda}^{-\gamma}\phi=0,$$
for every $\phi \in H^{1}_{0}(\Omega)$, which implies that $u_{\lambda}$ is a solution of $(P)_{\lambda}$. 

The Lemma \ref{l3} states that if $u_{\lambda}$ and $v_{\lambda}$ are solutions of problem $(P)_{\lambda}$, then  $u_{\lambda} \leq v_{\lambda}$ in $\Omega$ and $v_{\lambda} \leq u_{\lambda}$ in $\Omega$. Therefore, $u_{\lambda}= v_{\lambda}$ and as a consequence of this, problem $(P)_{\lambda}$ has a unique solution.     

Now, let us prove $a), b),c)$ and $d)$.\\
$a)$ It is enough show that $\displaystyle \liminf_{\lambda \uparrow \lambda_{\ast}}\Vert u_{\lambda}\Vert=\infty$. Suppose, reasoning by the contradiction, that $\displaystyle \liminf_{\lambda \uparrow \lambda_{\ast}}\Vert u_{\lambda}\Vert<\infty$. As a consequence of this, there exists a bounded sequence $\left\{u_{\lambda_{n}}\right\}\subset H^{1}_{0}(\Omega)$ with $\lambda_{n}\uparrow \lambda_{\ast}$. Therefore, we may assume that there exists a subsequence, still denoted $\left\{u_{\lambda_{n}}\right\}$, such that $u_{\lambda_{n}}\rightharpoonup u$ in $H^{1}_{0}(\Omega)$, for some $u\in H_{0}^{1}(\Omega)$. So, $u_{\lambda_{n}}\to u$ in $ L^{s}(\Omega)$ for all $s\in (0,2^{\ast})$ and $u_{\lambda_{n}}\to u\geq 0$ a.e. in $\Omega$. Since
\begin{equation}\label{l}
	\displaystyle \int_{\Omega} \nabla u_{\lambda_{n}}\nabla \phi=\displaystyle \int_{\Omega} a(x) u_{\lambda_{n}}^{-\gamma} \phi+\lambda_{n} \int_{\Omega}  f(u_{\lambda_{n}}) \phi,
\end{equation}
for all $\phi \in H^{1}_{0}(\Omega)$, it follows from Fatou's lemma that
$$\displaystyle \int_{\Omega} \nabla u\nabla \phi_{1}\geq \displaystyle \int_{\Omega} a(x) u^{-\gamma} \phi_{1}+\lambda \int_{\Omega}  f(u) \phi_{1}, $$
which implies that $u>0$ in $\Omega$. Now, from Lemma \ref{l3} we have that $u_{0}\leq  u_{\lambda_{n}}$ holds, which implies that $\vert  a(x) u_{\lambda_{n}}^{-\gamma} \phi \vert\leq \vert a(x) u_{0}^{-\gamma} \phi\vert\in L^{1}(\Omega)$ for all $\phi \in H_{0}^{1}(\Omega)$ and $n\in \mathbb{N}$. Therefore, it follows from Lebesgue's dominated convergence theorem that
\begin{equation}\label{ll}
	\displaystyle \int_{\Omega} a(x) u_{\lambda_{n}}^{-\gamma} \phi\to \displaystyle \int_{\Omega} a(x) u^{-\gamma} \phi ~\mbox{and}~\lambda_{n} \int_{\Omega}  f(u_{\lambda_{n}}) \phi \to \lambda_{\ast} \int_{\Omega}  f(u) \phi .
\end{equation}

Letting $n\to \infty$ in \eqref{l}, and using that $u_{\lambda_{n}}\rightharpoonup u$ in $H^{1}_{0}(\Omega)$ and \eqref{ll}, we get  
$$
\displaystyle \int_{\Omega} \nabla u\nabla \phi=\displaystyle \int_{\Omega} a(x) u^{-\gamma} \phi+\lambda_{\ast} \int_{\Omega}  f(u) \phi,
$$
for all $\phi \in H^{1}_{0}(\Omega)$. Thus, $u$ is a solution of $(P)_{\lambda}$ with $\lambda=\lambda_{\ast}$. But this contradicts the Lemma \ref{l0}. Therefore, $\displaystyle \limsup_{\lambda \uparrow \lambda_{\ast}}\Vert u_{\lambda}\Vert\geq \displaystyle \liminf_{\lambda \uparrow \lambda_{\ast}}\Vert u_{\lambda}\Vert=\infty$, which implies that $\Vert u_{\lambda}\Vert \to \infty$ as $\lambda \uparrow \lambda_{\ast}$. The proof of $a)$ is complete.

To prove $b)$ we need of the following lemma.
\begin{lemma}\label{l4}
	The function $m:(0,\lambda_{\ast})\to H_{0}^{1}(\Omega)$ defined by $m(\lambda)=u_{\lambda}$ is continuous.
\end{lemma}
\begin{proof}Let us fix a $\mu \in (0,\lambda_{\ast})$ and consider a sequence $\left\{\lambda_{n}\right\}\subset (0,\lambda_{\ast})$ such that $\lambda_{n} \to \mu$. To prove that   $u_{\lambda_{n}}\to  u_{\mu}$, we firstly show that the sequence $\left\{ u_{\lambda_{n}}\right\}$ is bounded in $H_{0}^{1}(\Omega)$. To this aim, we first remark that by Lemma \ref{l3} we have $u_{\lambda_{n}}\geq u_{0}$ in $\Omega$ for all $n\in \mathbb{N}$, and using \eqref{15} and $\lambda_{n}< \lambda_{\ast}$ we get 
	\begin{equation}\label{16}
		\Vert u_{\lambda_{n}} \Vert^{2}< \displaystyle \int_{\Omega}a\vert u_{0} \vert^{1-\gamma}+ \lambda_{\ast}\left(\theta+\epsilon\right)\displaystyle \int_{\Omega} \vert u_{\lambda_{n}}\vert^{2}+c \lambda_{\ast}\displaystyle \int_{\Omega}\vert u_{\lambda_{n}}\vert, 
	\end{equation}
	for each $\epsilon>0$ and some constant $c=c(\epsilon)>0$.
	
	Let us choose $\eta\in (0,\lambda_{\ast}$) satisfying $\eta> \lambda_{n}$ for all $n\in \mathbb{N}$. After this, by Lemma \ref{l3} the inequality $u_{\lambda_{n}}\leq u_{\eta}$ in $\Omega$ holds. This and \eqref{16} yields
	$$
	\Vert u_{\lambda_{n}} \Vert^{2}< \displaystyle \int_{\Omega}a(x)\vert u_{0} \vert^{1-\gamma}+ \lambda_{\ast}\left(\theta+\epsilon\right)\displaystyle \int_{\Omega} \vert u_{\eta}\vert^{2}+c \lambda_{\ast}\displaystyle \int_{\Omega}\vert u_{\eta}\vert,
	$$
	which implies that the sequence $\left\{u_{\lambda_{n}}\right\}$ is bounded in $H_{0}^{1}(\Omega)$. Therefore, we may assume that there is $0\leq u\in H_{0}^{1}(\Omega)$ such that $u_{\lambda_{n}}\rightharpoonup u$ in $H_{0}^{1}(\Omega)$ and $u_{\lambda_{n}}\to u$ in $L^{s}(\Omega)$, for all $s\in (0,2^{\ast})$. From these convergences and Lebesgue dominated convergence theorem we obtain that
	$$\displaystyle \lim_{n\to \infty} (u_{\lambda_{n}},u_{\lambda_{n}}-u)_{H}= \displaystyle \lim_{n\to \infty} \displaystyle \int_{\Omega}a(x)u_{\lambda_{n}}^{-\gamma}(u_{\lambda_{n}}-u)+\displaystyle \lim_{n\to \infty} \lambda_{n} \displaystyle \int_{\Omega}f(u_{\lambda_{n}})(u_{\lambda_{n}}-u)=0,$$
	which implies 
	$$\displaystyle \lim_{n\to \infty} \Vert  u_{\lambda_{n}}-u \Vert^{2}=\displaystyle \lim_{n\to \infty} (u_{\lambda_{n}},u_{\lambda_{n}}-u)_{H}-\displaystyle \lim_{n\to \infty} (u,u_{\lambda_{n}}-u)_{H} =0,$$
	holds, that is $u_{\lambda_{n}}\to u$ in $H_{0}^{1}(\Omega)$.
\end{proof}

Now we will prove $b)$. Let us fix $\lambda \in (0,\lambda_{\ast})$ and we show in what follows that 
$$	\frac{dI_{\lambda}(u_{\lambda})}{d\lambda}=-\displaystyle \int_{\Omega}F(u_{\lambda}).$$

Let $\mu\in (0,\lambda_{\ast})$. From \eqref{129} we have that
$$I_{\lambda}(u_{\lambda})=I_{\mu}(u_{\lambda})+(\mu-\lambda)\displaystyle \int_{\Omega}F(u_{\lambda})\geq I_{\mu}(u_{\mu})+(\mu-\lambda)\displaystyle \int_{\Omega}F(u_{\lambda})$$
and 
$$I_{\mu}(u_{\mu})=I_{\lambda}(u_{\mu})+(\lambda-\mu)\displaystyle \int_{\Omega}F(u_{\mu})\geq I_{\lambda}(u_{\lambda})+(\lambda-\mu)\displaystyle \int_{\Omega}F(u_{\mu}),$$
and as a consequence of these inequalities we obtain
\begin{equation}\label{14}
	-(\mu-\lambda)F(u_{\lambda})\geq I_{\mu}(u_{\mu})-I_{\lambda}(u_{\lambda})\geq -(\mu-\lambda)\displaystyle \int_{\Omega}F(u_{\mu}).
\end{equation}

Thus, \eqref{14} and Lemma \ref{l4} leads us to infer that
$$\displaystyle \liminf_{\mu \downarrow \lambda} \frac{I_{\mu}(u_{\mu})-I_{\lambda}(u_{\lambda})}{\mu-\lambda}=\displaystyle \limsup_{\mu \downarrow \lambda} \frac{I_{\mu}(u_{\mu})-I_{\lambda}(u_{\lambda})}{\mu-\lambda}=-\displaystyle \int_{\Omega}F(u_{\lambda})$$
and 
$$\displaystyle \liminf_{\mu \uparrow \lambda} \frac{I_{\mu}(u_{\mu})-I_{\lambda}(u_{\lambda})}{\mu-\lambda}=\displaystyle \limsup_{\mu \uparrow \lambda} \frac{I_{\mu}(u_{\mu})-I_{\lambda}(u_{\lambda})}{\mu-\lambda}=-\displaystyle \int_{\Omega}F(u_{\lambda}),$$
and therefore 
$$	\frac{dI_{\lambda}(u_{\lambda})}{d\lambda}=\displaystyle \lim_{\mu \to \lambda} \frac{I_{\mu}(u_{\mu})-I_{\lambda}(u_{\lambda})}{\mu-\lambda}$$ 
there exists and
$$	\frac{dI_{\lambda}(u_{\lambda})}{d\lambda}=-\displaystyle \int_{\Omega}F(u_{\lambda}).$$

Finally, we can apply the Lemma \ref{l4} to conclude that
the function $(0,\lambda_{\ast})\ni \lambda \to \int_{\Omega}F(u_{\lambda})$ is continuous and it follows that $I_{\lambda}(u_{\lambda})$ belongs to $C^{1}((0,\lambda_{\ast}),\mathbb{R})$. The proof of $b)$ is now complete.\\
$c)$ We claim that $u_{\lambda}=z+w$, where $z,w \in C^{1,\alpha}(\overline{\Omega})$ for some $\alpha\in (0,1)$. Indeed, by Theorem 3 in \cite{BN} there exist $\epsilon>0$ such that $u_{\lambda}(x)\geq \epsilon d(x)$ in $\Omega$. Hence,
$$0<a(x)u_{\lambda}^{-\gamma}(x)\leq  Cd^{\gamma-\eta}(x)d^{-\gamma}(x) \leq   Cd^{-\eta}(x),$$
for some constant $C>0$. So, from Lemma 2.1 in \cite{HA}  it follows that there exists $0<w\in C^{1,\alpha}(\overline{\Omega})$ (with $\alpha\in (0,1)$) such that
$$\displaystyle \int_{\Omega} \nabla w\nabla \phi=\displaystyle \int_{\Omega} a(x) u_{\lambda}^{-\gamma} \phi ,$$
for every $\phi \in H_{0}^{1}(\Omega)$, and this implies that
$$\displaystyle \int_{\Omega} \nabla u_{\lambda}\nabla \phi=\displaystyle \int_{\Omega} a(x) u_{\lambda}^{-\gamma} \phi+\lambda \int_{\Omega}  f(u_{\lambda}) \phi=\displaystyle \int_{\Omega} \nabla w\nabla \phi+\lambda \int_{\Omega}  f(u_{\lambda}) \phi, $$
that is
$$\displaystyle \int_{\Omega} \nabla (u_{\lambda}-w)\nabla \phi=\lambda \int_{\Omega}  f(u_{\lambda}) \phi= \lambda \int_{\Omega}  f((u_{\lambda}-w)+w) \phi$$
for every $\phi \in H_{0}^{1}(\Omega)$. So, the function $z=u_{\lambda}-w$ satisfies
\begin{equation*}
	\left\{
	\begin{array}{l}
		-\Delta z = g(x,z)~in ~ \Omega,\\
		z(x)=0~~on~~\partial \Omega,
	\end{array}
	\right.
\end{equation*}
where 
$$
g(x,t)=\left\{
\begin{array}{ccc}
\lambda f(t+w(x)), & \mbox{if} & t\geq 0 \\
\lambda f(w(x)), & \mbox{if} & t<0.\\
\end{array}
\right.
$$
Using \eqref{15} and the Young inequality we have
\begin{equation}\label{18}
	\mid g(x,t)\mid \leq c_{1}+c_{2}t\leq  c_{1}+c_{2}t^{p},
\end{equation}
for some constants $c_{1},c_{2}>0$ and $p\in (1,2^{\ast})$. From \eqref{18} we can use a usual bootstrap argument to get $z\in C^{1,\alpha}(\overline{\Omega})$, for some $\alpha\in (0,1)$. Since $u_{\lambda}=z+w$ it follows that $u_{\lambda}\in C^{1,\alpha}(\overline{\Omega})$ for some $\alpha\in (0,1)$.\\
$d)$  We observe that assumption $f(s)=s$ for all $s\in \mathbb{R}^{+}$ implies that $\lambda_{\ast}=\delta_{1}$. Using the Poincar\'e inequality we have
\begin{equation}\label{222}
	I_{\lambda}(u_{\lambda})\geq \frac{1}{2}\left(1-\frac{\lambda}{\delta_{1}}\right)\Vert u_{\lambda}\Vert^{2}-\frac{1}{1-\gamma}\displaystyle \int_{\Omega}a(x)\vert u_{\lambda} \vert^{1-\gamma}>0,
\end{equation} 
for every $\lambda \in \left(0,\lambda_{\ast}\right)$. 

By virtude of \eqref{222} and \eqref{10} it follows that
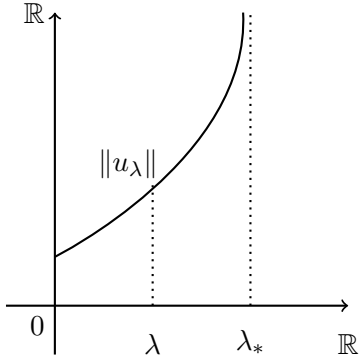
$$0<	I_{\lambda}(u_{\lambda})\leq J_{\lambda}(t_{\lambda}(\phi_{1})\phi_{1})=\displaystyle -\left(\frac{1+\gamma}{2(1-\gamma)}\right)\left(\displaystyle \int a(x)\vert \phi_{1}\vert^{1-\gamma}\right)\left(\frac{H_{\lambda}(\phi_{1})}{\displaystyle \int a(x)\vert \phi_{1}\vert^{1-\gamma}}\right)^{\frac{\gamma-1}{1+\gamma}}, $$
and since $1<\gamma$ and
$$H_{\lambda}(\phi_{1})=(\lambda_{1}-\lambda )\displaystyle \int \vert \phi_{1}\vert^{2}\to 0$$ 
as $\lambda \to \lambda_{\ast}$, we have that 
$$0\leq \displaystyle \liminf_{\lambda \uparrow \lambda_{\ast}}	I_{\lambda}(u_{\lambda})\leq 	\displaystyle \limsup_{\lambda \uparrow \lambda_{\ast}}	I_{\lambda}(u_{\lambda})\leq 0,$$
which implies that $I_{\lambda}(u_{\lambda})\to 0$ as $\lambda \to \lambda_{\ast}$. The prove of $d)$ is complete.
\section{Proof of Theorem \ref{T2}}
Let us prove the Theorem \ref{T2}. Let $u_{\lambda}$ be as in the Lemma \ref{l1} and let us prove that $u_{\lambda}$ is solution of $(P)_{\lambda}$. To this aim, fisrt we consider $\phi \in H^{1}_{0}(\Omega)$ such that $\phi \geq 0$ in $\Omega$ and $\epsilon>0$. Since $ u_{\lambda}+\epsilon \phi\in D^{+}$, it follow that
$$\frac{1}{2}\Vert u_{\lambda}+\epsilon \phi\Vert^{2}-\frac{1}{2}\Vert u_{\lambda}\Vert^{2}-\lambda \displaystyle \int_{\Omega}F(u_{\lambda}+\epsilon \phi)-\displaystyle F(u_{\lambda})\geq \dfrac{1}{1-\gamma}\displaystyle \int_{\Omega}a(x)\vert u_{\lambda}+\epsilon \phi \vert^{1-\gamma} -\dfrac{1}{1-\gamma}a(x)\vert u_{\lambda} \vert^{1-\gamma}.$$

Thus, dividing the last inequality by $\epsilon$ and passing to the liminf as $\epsilon \to 0$, by Fatou's Lemma we have
$$
	\displaystyle \int_{\Omega} \nabla u_{\lambda} \nabla \phi  - \lambda\displaystyle \int_{\Omega} f(u_{\lambda})\phi \geq \displaystyle \int_{\Omega} u_{\lambda}^{-\gamma}\phi,
$$
for every $\phi \in H^{1}_{0}(\Omega)$ with $\phi \geq 0$. 

Now, since $tu_{\lambda}\in D^{+}$ for every $t>0$, we have that the function $\psi\in C^{1}((0,\infty),\mathbb{R})$ defined by $\psi(t)=I_{\lambda}(tu_{\lambda})$ has a global minimum at $t=1$, and therefore
$$
	0=\psi^{\prime}(1)=\Vert u_{\lambda}\Vert^{2}-\lambda \displaystyle \int_{\Omega}f(u_{\lambda})u_{\lambda}-\displaystyle \int_{\Omega}a(x)\vert u \vert^{1-\gamma}.
$$

Finally, following the proof of Theorem \ref{T1} we can prove that $u_{\lambda}$ is a solution of $(P)_{\lambda}$. Moreover, by Lemma \ref{l3} we have that $u_{\lambda}$ is unique.

Now, let us prove $a), b),c),d)$ and $e)$.\\
$a)$ First, we note that 
$$H_{\lambda}(\phi_{1})=\left(\lambda_{1}-\lambda \theta\right)\displaystyle \int_{\Omega}\phi_{1}^{2},$$
and using \eqref{10} this implies that $J_{\lambda}(t_{\lambda}(\phi_{1})\phi_{1})\to -\infty$. Hence, by \eqref{11} we have
$$-\infty\leq \displaystyle \lim_{\lambda \uparrow \lambda_{\ast}}I_{\lambda}(u_{\lambda})\leq \displaystyle \lim_{\lambda \uparrow \lambda_{\ast}}I_{\lambda}(t_{\lambda}(\phi_{1})\phi_{1}) \leq \displaystyle \lim_{\lambda \uparrow \lambda_{\ast}}J_{\lambda}(t_{\lambda}(\phi_{1})\phi_{1})=-\infty,$$
that is $\displaystyle \lim_{\lambda \uparrow \lambda_{\ast}}I_{\lambda}(u_{\lambda})=-\infty$.

$b)$ It is enough to prove that $\displaystyle \liminf_{\lambda \uparrow \lambda_{\ast}}\Vert u_{\lambda}\Vert=\infty$. Suppose, reasoning by the contradiction, that $\displaystyle \liminf_{\lambda \uparrow \lambda_{\ast}}\Vert u_{\lambda}\Vert<\infty$. As a consequence of this, there exists a
bounded sequence $\left\{u_{\lambda_{n}}\right\}\subset H^{1}_{0}(\Omega)$ with $\lambda_{n}\uparrow \lambda_{\ast}$. Therefore, we may assume that there exists a subsequence, still denoted $\left\{u_{\lambda_{n}}\right\}$, such that $u_{\lambda_{n}}\rightharpoonup u$ in $H^{1}_{0}(\Omega)$ for some $u\in H_{0}^{1}(\Omega)$.  So, $u_{\lambda_{n}}\to u$ in $ L^{s}(\Omega)$ for all $s\in (0,2^{\ast})$ and $u_{\lambda_{n}}\to u\geq 0$ a.e. in $\Omega$. Since $u\longmapsto \Vert u \Vert^{2}$ is  weakly lower semicontinuous, from item $a)$ we obtain that
$$-\infty < I_{\lambda_{\ast}}(u)\leq \displaystyle \liminf_{\lambda_{n} \uparrow \lambda_{\ast}}I_{\lambda_{n}}(u_{\lambda_{n}})=-\infty,$$
which is an absurd. Therefore, $\displaystyle \limsup_{\lambda \uparrow \lambda_{\ast}}\Vert u_{\lambda}\Vert\geq \displaystyle \liminf_{\lambda \uparrow \lambda_{\ast}}\Vert u_{\lambda}\Vert=\infty$, which implies that $\Vert u_{\lambda}\Vert \to \infty$ as $\lambda \uparrow \lambda_{\ast}$. The proof of $b)$ is complete.

$c)$ Let $\lambda<\mu$. By Lemma \ref{l3} we have $u_{\lambda}\leq u_{\mu}$ in $\Omega$. Hence,
$$\Vert u_{\lambda}\Vert^{2}=\displaystyle \int_{\Omega} a(x)\vert u_{\lambda}\vert^{1-\gamma}+\lambda \displaystyle \int_{\Omega} f(u_{\lambda})u_{\lambda}\leq \displaystyle \int_{\Omega} a(x)\vert u_{\mu}\vert^{1-\gamma}+\mu \displaystyle \int_{\Omega} f(u_{\mu})u_{\mu}=\Vert u_{\mu}\Vert^{2},$$
that is $\Vert u_{\lambda}\Vert\leq \Vert u_{\mu}\Vert$.\\

$d)$ The proof of $d)$ follows the same idea of the proof of $b)$ of the Theorem \ref{T1}.

$e)$ As in $c)$ of Theorem \ref{T1} we let us prove that $u_{\lambda}=z+w$, where $z,w \in C^{1,\alpha}(\overline{\Omega})$ for some $\alpha\in (0,1)$. First, we know from Theorem 3 in \cite{BN} that there exists $\epsilon>0$ such that $u_{\lambda}\geq \epsilon d(x)$. Hence, since $a\in L^{\infty}(\Omega)$ we have
$$0<a(x)u_{\lambda}^{-\gamma}(x)\leq  Cd^{-\gamma}(x),$$
for some constant $C>0$. Now the proof follows the same idea of the proof of $c)$ of Theorem \ref{T1}. The proof is complete.

\end{document}